\documentclass[12pt,reqno]{amsart}

\usepackage[T1]{fontenc}
\usepackage{lmodern}
\usepackage{microtype}
\usepackage{mathtools}
\usepackage{amssymb}
\usepackage{mathrsfs}
\usepackage{enumitem}
\usepackage{booktabs}
\usepackage{array}
\usepackage{xcolor}
\usepackage[colorlinks=true,linkcolor=blue!55!black,citecolor=green!45!black,
  urlcolor=blue!65!black]{hyperref}
\usepackage[nameinlink,capitalise,noabbrev]{cleveref}

\hypersetup{
  pdftitle={Simple Cartan-Free Modules of Finite Rank over the Super-Virasoro Algebras},
  pdfauthor={Yang Chen, Yongqi Luo, and Yueqiang Zhao},
  pdfsubject={A uniform arbitrary-rank classification in the Ramond and Neveu--Schwarz sectors}
}

\setlist[enumerate]{label=\textup{(\roman*)},leftmargin=2.2em}
\numberwithin{equation}{section}

\newtheorem{theorem}{Theorem}[section]
\newtheorem{proposition}[theorem]{Proposition}
\newtheorem{lemma}[theorem]{Lemma}
\newtheorem{corollary}[theorem]{Corollary}
\theoremstyle{definition}
\newtheorem{definition}[theorem]{Definition}

\theoremstyle{remark}
\newtheorem{remark}[theorem]{Remark}

\newcommand{\C}{\mathbb C}
\newcommand{\Z}{\mathbb Z}
\newcommand{\Ztwo}{\mathbb Z_2}
\newcommand{\D}{\mathcal D}
\newcommand{\A}{\mathcal A}
\newcommand{\K}{\mathcal K}

\title[Simple finite-rank Cartan-free modules]
{Simple Cartan-Free Modules of Finite Rank over the
Super-Virasoro Algebras}

\author{Yang Chen}
\address{Y. Chen: School of Science, Jimei University, Xiamen, Fujian, 361021, P. R. China}
\email{chenyang1729@hotmail.com}

\author{Yongqi Luo}
\address{Y. Luo: School of Science, Jimei University, Xiamen, Fujian, 361021, P. R. China}
\email{luoyongqi777@outlook.com}

\author{Yueqiang Zhao}
\address{Y. Zhao: School of Mathematics and Statistics, Xinyang Normal University, Xinyang, 464000, P. R. China}
\email{yueqiangzhao@163.com}
\thanks{This work was supported by the National Natural Science Foundation
of China (Grant No.~12301040) and the Nanhu Scholars Program for Young
Scholars of Xinyang Normal University (Grant No.~2021030).}
\date{}


\begin{document}

\begin{abstract}
We classify nontrivial simple super-Virasoro modules finitely
generated over \(\C[L_0]\) using simple Laurent Weyl algebra
modules finite free over \(\C[D]\).  The classified modules are
finite Cartan-free with equal parity ranks.  We also prove
finite length without simplicity and construct explicit
nonsplit self-extensions.
\end{abstract}

\maketitle

\section{Introduction}
\label{sec:unified-introduction}

The super-Virasoro algebras, also called the \(N=1\) superconformal
algebras, are infinite-dimensional Lie superalgebras extending the
Virasoro algebra by odd generators.  They occur as symmetry algebras
in superconformal field theory and superstring theory, and their
representations connect these subjects with the structure theory of
infinite-dimensional Lie superalgebras.  There are two sectors, the
Ramond and Neveu--Schwarz algebras, distinguished by the integer and
half-integer indices of their odd generators.

Throughout, all vector spaces and algebras are over \(\C\).
All modules over Lie superalgebras are \(\Ztwo\)-graded, and all
homomorphisms are even.  For \(\epsilon\in\{0,1/2\}\), let
\(\widehat{\mathfrak g}^{\,\epsilon}\) be the Lie superalgebra with
even basis \(\{L_m,C:m\in\Z\}\), odd basis
\(\{G_p:p\in\Z+\epsilon\}\), and relations
\begin{align*}
 [L_m,L_n]&=(n-m)L_{m+n}
  +\delta_{m+n,0}\frac{m^3-m}{12}C,
\\
 [L_m,G_p]&=\left(p-\frac m2\right)G_{m+p},
\\
 [G_p,G_q]&=-2L_{p+q}
  +\delta_{p+q,0}\frac{4p^2-1}{12}C,
\\
 [C,\widehat{\mathfrak g}^{\,\epsilon}]&=0.
\end{align*}
The cases \(\epsilon=0\) and \(\epsilon=1/2\) are the Ramond
and Neveu--Schwarz algebras, respectively.  Write
\(\mathfrak g^\epsilon=\widehat{\mathfrak g}^{\,\epsilon}/\C C\).

Let \(\mathfrak h_{\bar0}=\C L_0\) denote the Cartan subalgebra
of the even part modulo the center, so
\(U(\mathfrak h_{\bar0})=\C[L_0]\).
A \(\mathfrak g^\epsilon\)-module or \(\widehat{\mathfrak g}^{\,\epsilon}\)-module
\(M=M_{\bar0}\oplus M_{\bar1}\) is called
\emph{finite Cartan-free} if it is free of finite rank over
\(\C[L_0]\).  This is equivalent to both parity components being
finite free.
At a fixed central character \(C=c\), this polynomial algebra is
identified with \(U(\C L_0\oplus\C C)/(C-c)\).
The parity ranks are equal for every finite Cartan-free module
at a fixed central character by \cref{lem:automatic-balance}.
We use the parity change
\(\Pi\) with \(x\cdot\Pi v=(-1)^{|x|}\Pi(xv)\) for homogeneous \(x\).

The simple Harish--Chandra modules over the super-Virasoro algebras
were classified by Su \cite{Su1995}.  New proofs using \(A\)-covers
and the theory of jet modules \cite{Billig2007} were given by
Cai--Liu--L\"u \cite{CaiLiuLu2021}
and Cai--L\"u \cite{CaiLu2022}.  For Witt superalgebras, related
methods were used independently by Xue--L\"u \cite{XueLu2021} and
Billig--Futorny--Iohara--Kashuba \cite{BilligFutornyIoharaKashuba2020}
to classify simple weight modules with finite-dimensional weight spaces
relative to the even Euler operators.
Beyond the weight setting, a natural
class consists of modules that are free over the enveloping algebra
of a Cartan subalgebra.  Yang--Yao--Xia \cite{YangYaoXia2020}
classified the Ramond modules free of rank one over
\(U(\C L_0\oplus\C G_0)\)
and the Neveu--Schwarz modules free of rank two over
\(\C[L_0]\).  The corresponding problem in arbitrary
finite rank was solved for the Witt algebra by Guo--Liu--L\"u--Zhao
\cite{GuoLiuLuZhao2020}, who classified simple modules finitely
generated over the enveloping algebra of its Cartan subalgebra.
L\"u--Zhao \cite{LuZhao2014} constructed Virasoro modules from
simple Weyl algebra modules.
Chen--Dai--Liu--Pei \cite{ChenDaiLiuPei2024}
extended this construction to the super-Virasoro algebras and
established simplicity and isomorphism criteria.
Guo--Li--Liu \cite{GuoLiLiu2026} constructed Neveu--Schwarz modules
from Weyl superalgebra modules and restricted modules over the
positive-part subalgebra of the Neveu--Schwarz algebra.

Let \(\D\) be the Weyl algebra generated by \(t^{\pm1}\) and \(D\)
subject to the relation \(Dt-tD=t\).  For a \(\D\)-module \(P\) and \(b\in\C\),
let \(H_{\epsilon,b}(P)=P\oplus\overline P\) be the super-Virasoro
module constructed in \cite{ChenDaiLiuPei2024}, where
\(\overline P\) is an odd copy of \(P\).
In this paper, we prove the following classification theorem.

\begin{theorem}
\label{thm:unified-main}
For \(\epsilon\in\{0,1/2\}\), the nontrivial simple finite
Cartan-free modules over \(\mathfrak g^\epsilon\) or
\(\widehat{\mathfrak g}^{\,\epsilon}\) are exactly
\[
 \Pi^\delta H_{\epsilon,b}(P),\qquad
 \delta\in\Ztwo,\quad b\in\C\setminus\{\tfrac12\},
\]
where \(P\) is a simple \(\D\)-module finite free over \(\C[D]\)
and \(C\) acts trivially in the centrally extended case.
The parameters \(b,\delta\) and the isomorphism class of \(P\)
are uniquely determined.
\end{theorem}

\begin{corollary}
\label{cor:finite-generated-simple}
Every simple \(\mathfrak g^\epsilon\)- or
\(\widehat{\mathfrak g}^{\,\epsilon}\)-module finitely generated
over \(\C[L_0]\) is either a one-dimensional trivial module,
concentrated in even or odd degree, or one of the modules in
\cref{thm:unified-main}.
\end{corollary}

The proof combines the weighting functor
\cite{Nilsson2016} with the \(A\)-cover method of Billig--Futorny
\cite{BilligFutorny2016} to obtain \(A\)-covers finitely generated
over \(\C[L_0]\).  The algebra factorizations in
\cite{CaiLiuLu2021,CaiLu2022} and the tensor-product lemma in
\cite{XueLu2021} reduce the problem to simple Laurent Weyl
superalgebra modules and one-dimensional modules for the corresponding Lie
superalgebras.  The simplicity and isomorphism criteria of
\cite{ChenDaiLiuPei2024} complete the
classification.  Finally, we show that every module finitely generated over
\(\C[L_0]\) has finite length.  For the functors
\(H_{\epsilon,b}\) of \cite{ChenDaiLiuPei2024}, we determine
the even homomorphisms between their values and construct
explicit nonsplit self-extensions of the classified simple
modules.

\section{Modules from the Weyl algebra}
\label{sec:unified-Weyl-H}

Let
\begin{equation*}
 \D=\C\langle t^{\pm1},D\rangle/(Dt-tD-t)
\end{equation*}
be the algebra of differential operators on \(\C[t^{\pm1}]\).
Setting \(\partial_t=t^{-1}D\) gives
\(\partial_t t-t\partial_t=1\), so \(\D\) is the localization of
the first Weyl algebra at \(t\).  Our generator \(D\) corresponds
to \(\partial=t\,d/dt\) in
\cite{LuZhao2014}.  Thus
\begin{equation*}
 Dt^m=t^m(D+m),\qquad t^mf(D)=f(D-m)t^m
\end{equation*}
for \(m\in\Z\) and \(f\in\C[D]\).

\begin{lemma}
\label{lem:weyl-finite-length}
Let \(E\) be a \(\D\)-module finitely generated over \(\C[D]\).
Then \(E\) is free of finite rank over \(\C[D]\), and its
composition length as a \(\D\)-module is at most
\(\operatorname{rank}_{\C[D]}E\).
\end{lemma}

\begin{proof}
The freeness assertion is the one-variable case of the result for
Laurent Weyl modules in \cite{GuoLiuLuZhao2020}.
Every nonzero \(\D\)-subquotient of \(E\) is finitely generated
over \(\C[D]\), and hence free of positive rank.
The length bound follows by additivity of rank.
\end{proof}

We consider simple \(\D\)-modules
that are finite free over \(\C[D]\) (see \cite{Block1981} for
background on simple Weyl modules).

\begin{definition}[{\cite[Proposition~3.4 and Theorem~3.6]{ChenDaiLiuPei2024}}]
\label{def:unified-H}
For \(\epsilon\in\{0,1/2\}\), \(b\in\C\), and a
\(\D\)-module \(P\), put \(H_{\epsilon,b}(P)=P\oplus\overline P\) with
\(P\) even and \(\overline P\) odd.  The action is given by
\begin{align}
 L_m v&=t^m\bigl(D+(m-2\epsilon)b\bigr)v,\notag\\
 L_m\overline v
   &=\overline{t^m\bigl(D+(m-2\epsilon)b+\tfrac{m+2\epsilon}{2}\bigr)v},
 \label{eq:H-even-components}\\
 G_pv&=\overline{t^{p-\epsilon}(D+2(p-\epsilon)b)v},\notag\\
 G_p\overline v&=-t^{p+\epsilon}v,
 \label{eq:H-odd-components}
\end{align}
where \(m\in\Z\), \(p\in\Z+\epsilon\), and \(v\in P\).
The central element acts by zero.
\end{definition}

\begin{lemma}
\label{lem:automatic-balance}
Let \(M\) be a finite Cartan-free module with central charge
\(c\) in the centrally extended case.  Its two parity components
have equal rank over \(\C[L_0]\).
\end{lemma}

\begin{proof}
Take \(c=0\) in the centerless case.
In the Ramond case, \(G_0^2=-(L_0+c/24)\), so \(G_0\) is
injective because \(M\) is torsion-free over \(\C[L_0]\).
In the Neveu--Schwarz case, if \(0\ne v\in\ker L_1\), the
commutation relations and \(\C[L_0]\)-torsion-freeness imply that
the vectors \(L_{-1}^nv\), \(n\ge0\), are linearly independent
over \(\C[L_0]\), contradicting the finite rank of \(M\).
Thus \(L_1\) is injective, and \(G_{1/2}^2=-L_1\) implies that
\(G_{1/2}\) is injective as well.
Since \(G_\epsilon f(L_0)=f(L_0-\epsilon)G_\epsilon\), the injective
operator \(G_\epsilon\) preserves \(\C[L_0]\)-linear independence.
Applying this to both parity components shows that their ranks are equal.
\end{proof}

\begin{remark}
\label{rem:H-Cartan-rank}
Suppose that \(P\) has a \(\C[D]\)-basis \(v_1,\ldots,v_r\).
By \eqref{eq:H-even-components} with
\(m=0\), the sets \(\{v_1,\ldots,v_r\}\) and
\(\{\overline{v_1},\ldots,\overline{v_r}\}\) form \(\C[L_0]\)-bases
of the even and odd components of \(H_{\epsilon,b}(P)\), respectively.
For the centerless Ramond algebra, put
\(\mathfrak h_R=\C L_0\oplus\C G_0\).
Since \(G_0^2=-L_0\) and \(G_0\overline{v_i}=-v_i\), the vectors
\(\overline{v_1},\ldots,\overline{v_r}\) form a
\(U(\mathfrak h_R)\)-basis of \(H_{0,b}(P)\).
After parity change, this basis is even.
\end{remark}

\begin{proposition}
\label{prop:unified-H-simple}
Let \(P\) be a simple \(\D\)-module finite free over \(\C[D]\).
\begin{enumerate}
\item \(H_{\epsilon,b}(P)\) is simple if and only if
\(b\ne1/2\).
\item The module \(H_{\epsilon,1/2}(P)\) contains the simple
Cartan-free submodule
\begin{equation}
 K_\epsilon(P)=P\oplus\overline{DP}\simeq\Pi H_{\epsilon,0}(P),
 \label{eq:exceptional-parity-zero}
\end{equation}
with purely odd trivial quotient \(\Pi(P/DP)\).
Moreover, \(H_{\epsilon,1/2}(P)\) has no nontrivial simple quotient.
\end{enumerate}
\end{proposition}

\begin{proof}
For \textup{(i)}, finite \(\C[D]\)-freeness excludes the natural
Laurent module \(\C[t^{\pm1}]\) and gives \(DP\ne P\), so
\cite[Theorem~4.5]{ChenDaiLiuPei2024} applies.
For \textup{(ii)}, the simplicity of \(K_\epsilon(P)\) and
\eqref{eq:exceptional-parity-zero} follow from
\cite[proof of Lemma~4.4]{ChenDaiLiuPei2024}.
Its Cartan freeness follows from \Cref{rem:H-Cartan-rank}, and
the quotient is determined by the defining actions.
Since \(G_{-\epsilon}\overline v=-v\), every nonzero submodule
of \(H_{\epsilon,1/2}(P)\) meets \(K_\epsilon(P)\) nontrivially
and hence contains it.  Thus every simple quotient is trivial.
\end{proof}

\begin{corollary}
\label{cor:unified-H-uniqueness}
Let \(P,P'\) be simple \(\D\)-modules finite free over
\(\C[D]\).  For \(b,b'\ne1/2\) and
\(\delta,\delta'\in\Ztwo\), we have
\[
 \Pi^\delta H_{\epsilon,b}(P)
 \simeq\Pi^{\delta'}H_{\epsilon,b'}(P')
\]
if and only if \(\delta=\delta'\), \(b=b'\), and \(P\simeq P'\).
\end{corollary}

\begin{proof}
By \eqref{eq:H-odd-components}, \(G_\epsilon\) maps the odd component
of \(H_{\epsilon,b}(P)\) onto its even component, but not conversely,
since \(DP\ne P\).  The same holds for \(H_{\epsilon,b'}(P')\).
Thus an isomorphism as above forces \(\delta=\delta'\).
The remaining assertions follow from
\cite[Lemmas~4.6--4.7]{ChenDaiLiuPei2024}.
\end{proof}

We recall a matrix description of \(\D\)-modules that are finite
free over \(\C[D]\).
Put
\[
 \sigma(f(D))=f(D-1),\qquad f(D)\in\C(D),
\]
and apply \(\sigma\) entrywise to vectors and matrices.
Thus \(\D\) is the skew Laurent polynomial ring
\(\C[D][t^{\pm1};\sigma]\).
By \cite[Proposition~5.1 and Lemma~5.2]{LopesNilsson2025},
every \(\D\)-module free of finite rank \(r\) over \(\C[D]\) is
isomorphic to \(P_A=\C[D]^r\) with action
\begin{equation}
 t\cdot v=A\sigma(v),\qquad
 t^{-1}\cdot v=\sigma^{-1}(A^{-1}v),
\label{eq:coefficient-matrix-action}
\end{equation}
for some \(A\in\operatorname{GL}_r(\C[D])\).
Moreover, \(P_A\simeq P_B\) if and only if
\begin{equation*}
 B=SA\,\sigma(S)^{-1},
\end{equation*}
for some \(S\in\operatorname{GL}_r(\C[D])\).
In the notation of \cite{LopesNilsson2025}, we take \(a=1\),
\(x=t\), \(y=t^{-1}\), and the automorphism to be \(\sigma^{-1}\).
\(P_A\) is simple if and only if the difference module
\((\C(D)^r,\Phi_A)\) is irreducible, where \(\Phi_A(v)=A\sigma(v)\).
Indeed, if \(W\) is a proper nonzero \(\C(D)\)-subspace
with \(\Phi_A(W)=W\), then \(W\cap\C[D]^r\) is a proper nonzero
\(\D\)-submodule, since both \(A\) and \(A^{-1}\) have polynomial
entries.  Conversely, for a \(\D\)-submodule
\(0\ne N\subsetneq P_A\), the subspace \(\C(D)N\) is nonzero and
invariant.  By \Cref{lem:weyl-finite-length}, the nonzero quotient
\(P_A/N\) is free of positive rank over \(\C[D]\), so
\(\C(D)N\) is proper.

\section{Weighting and \texorpdfstring{\(A\)-covers}{A-covers}}
\label{sec:uniform-core}

\subsection{Weighting and annihilators}

For a finite Cartan-free \(\mathfrak g^\epsilon\)-module \(M\) and
\(\lambda\in\mathbb C\), put
\begin{equation*}
 \mathcal W^\epsilon_\lambda(M)
   :=\bigoplus_{a\in\lambda+\mathbb Z+\epsilon\mathbb Z}
       M/(L_0-a)M.
\end{equation*}
The weights thus range over \(\lambda+\mathbb Z\) in the Ramond case
and \(\lambda+(1/2)\mathbb Z\) in the Neveu--Schwarz case.
The action on the quotient spaces is given by
\begin{equation*}
 \begin{aligned}
 L_m\bigl(v+(L_0-a)M\bigr)&=L_mv+(L_0-a-m)M,\\
 G_p\bigl(v+(L_0-a)M\bigr)&=G_pv+(L_0-a-p)M.
 \end{aligned}
\end{equation*}
The defining brackets make these actions well defined, just as in
\cite[Proposition~8]{Nilsson2016}.  The same construction applies to
the central extension, with \(C\) acting on each quotient as on \(M\).
Since \(M\) is finite free over \(\mathbb C[L_0]\), each weight space
has dimension
\begin{equation*}
 \dim\mathcal W^\epsilon_\lambda(M)_a
   =\operatorname{rank}_{\mathbb C[L_0]}M,
 \qquad a\in\lambda+\mathbb Z+\epsilon\mathbb Z.
\end{equation*}
Thus \(\mathcal W^\epsilon_\lambda(M)\) is uniformly bounded.

\begin{proposition}
\label{prop:uniform-central-charge}
Let \(M\) be a nontrivial simple
\(\widehat{\mathfrak g}^{\,\epsilon}\)-module which is finite Cartan-free.
Then \(C\) acts trivially on \(M\).
\end{proposition}

\begin{proof}
Since \(M\) is countable dimensional, Dixmier's lemma
\cite{Dixmier1977}, applied to even endomorphisms, gives
\(C|_M=c\,\mathrm{id}_M\) for some \(c\in\mathbb C\).
A nonzero cyclic graded submodule of
\(\mathcal W^\epsilon_\lambda(M)\) has a simple quotient, which is
uniformly bounded and has the same central charge \(c\).
The Harish--Chandra classification
\cite{Su1995,CaiLiuLu2021,CaiLu2022} therefore gives \(c=0\).
\end{proof}

We henceforth work with the centerless algebra \(\mathfrak g^\epsilon\).
Throughout this section, indices are understood to lie in the ranges
of the corresponding generators.
For \(X,Y\in\{L,G\}\), \(N\geq0\), and indices \(k,p\),
we consider the following finite-difference operators in
\(U(\mathfrak g^\epsilon)\):
\begin{equation*}
 \Omega^{(N)}_{X,Y}(k,p)
   :=\sum_{i=0}^{N}(-1)^i\binom Ni
       X_{k-i}Y_{p+i}.
\end{equation*}
Every summand has \(L_0\)-degree \(k+p\).  Pascal's identity gives
\begin{equation}
 \Omega^{(N+1)}_{X,Y}(k,p)
 =\Omega^{(N)}_{X,Y}(k,p)
  -\Omega^{(N)}_{X,Y}(k-1,p+1).
 \label{eq:uniform-higher-difference}
\end{equation}
Hence annihilation at order \(N\) implies annihilation at every higher
order.

\begin{lemma}
\label{lem:uniform-four-differentiators-weight}
There exists \(N\geq2\) such that
\begin{equation*}
 \Omega^{(N)}_{X,Y}(k,p)
    \mathcal W^\epsilon_\lambda(M)=0
\end{equation*}
for all \(X,Y\in\{L,G\}\) and indices \(k,p\).
\end{lemma}

\begin{proof}
The restriction of \(\mathcal W^\epsilon_\lambda(M)\) to the Witt
algebra spanned by the \(L_n\) is uniformly bounded and supported on
at most two cosets of \(\mathbb Z\).  Every nonzero subquotient
supported on a coset \(\alpha+\mathbb Z\) has a simple subquotient
and hence, by Mathieu's classification \cite{Mathieu1992}, a nonzero
weight space at \(\alpha\) if \(\alpha\notin\mathbb Z\), or at
\(0\) or \(1\) otherwise.  By additivity, the dimensions at these
finitely many weights bound the length of every strict submodule
chain.  Thus the restriction has finite composition length.

Applying
\cite[Lemmas~3.6--3.7]{CaiLiuLu2021} in the Ramond case and
\cite[Lemmas~2.2(1) and~4.3(1)]{CaiLu2022} in the Neveu--Schwarz
case, we find that the operators \(\Omega^{(N)}_{L,L}(k,p)\) and
\(\Omega^{(N)}_{G,L}(k,p)\) annihilate
\(\mathcal W^\epsilon_\lambda(M)\) at sufficiently large orders.  By
\eqref{eq:uniform-higher-difference}, we may choose a common order
\(N\geq2\).
For \(k\in\mathbb Z+\epsilon\) and \(p\in\mathbb Z\), the
commutation relations give
\begin{equation*}
 \Omega^{(N)}_{G,L}(k,p)
  =(-1)^N\Omega^{(N)}_{L,G}(p+N,k-N).
\end{equation*}
Finally, for \(k,s\in\mathbb Z+\epsilon\), we have
\[
 \frac32\Omega^{(N)}_{G,G}(k,s)
 =[\Omega^{(N)}_{G,L}(k,0),G_s]
 -[\Omega^{(N)}_{G,L}(k,1),G_{s-1}]=0
\]
on \(\mathcal W^\epsilon_\lambda(M)\).
\end{proof}

\begin{lemma}
\label{lem:uniform-differentiator-lift}
There exists \(N\geq2\) such that
\begin{equation*}
 \Omega^{(N)}_{X,Y}(k,p)M=0
\end{equation*}
for all \(X,Y\in\{L,G\}\) and indices \(k,p\).
\end{lemma}

\begin{proof}
Choose \(N\) as in \Cref{lem:uniform-four-differentiators-weight}.
Since \(\Omega^{(N)}_{X,Y}(k,p)\) is homogeneous of
\(L_0\)-degree \(k+p\) and annihilates
\(\mathcal W^\epsilon_\lambda(M)\), its action on the summands of
weights \(\lambda+n\), \(n\in\mathbb Z\), gives
\begin{equation*}
 \Omega^{(N)}_{X,Y}(k,p)v
 \in\bigcap_{n\in\mathbb Z}(L_0-\lambda-n-k-p)M,
 \qquad v\in M.
\end{equation*}
Since \(M\) is free over \(\mathbb C[L_0]\), this intersection is zero.
\end{proof}

\subsection{\texorpdfstring{\(A\)-covers}{A-covers}}

We use the \(A\)-cover construction of \cite{BilligFutorny2016}.  Let
\[
 \mathcal A=\mathbb C[t^{\pm1}]\otimes\Lambda(\xi),
 \qquad |t|=\bar0,\quad |\xi|=\bar1.
\]
In the standard contact realization, \(\mathfrak g^\epsilon\) consists
of the superderivations preserving the contact distribution
\(\mathcal A(\partial_\xi+t^{1-2\epsilon}\xi\partial_t)\)
\cite[Section~2.2]{GargoubiMellouliOvsienko2007}.
The corresponding action on \(\mathcal A\) is
\begin{align*}
 L_m\circ f
  &=t^{m+1}\partial_tf
    +\frac{m+2\epsilon}{2}t^m\xi\partial_\xi f,
\\
 G_p\circ f
  &=t^{p-\epsilon+1}\xi\partial_tf
    -t^{p+\epsilon}\partial_\xi f.
\end{align*}
Contact vector fields are canonically identified with sections of
the contact line bundle
\cite[Corollary~3.2 and Proposition~3.4]{GargoubiMellouliOvsienko2007}.
Under this identification, multiplication by functions induces the
following \(\mathcal A\)-module structure:
\begin{equation}
 \begin{aligned}
  t^iL_m&=L_{m+i},&
  t^iG_p&=G_{p+i},\\
  \xi L_m&=\frac12G_{m+\epsilon},&
  \xi G_p&=0,
 \end{aligned}
 \qquad i,m\in\mathbb Z,\quad p\in\mathbb Z+\epsilon.
 \label{eq:uniform-A-on-g}
\end{equation}
The adjoint action corresponds to the Lie derivative on these sections.
For homogeneous elements \(X,Y\in\mathfrak g^\epsilon\) and
\(a\in\mathcal A\), the graded Leibniz rule gives
\[
 [X,aY]=(X\circ a)Y+(-1)^{|X||a|}a[X,Y].
\]
An \(\mathcal A\mathfrak g^\epsilon\)-module is a module over the
semidirect product \(\mathfrak g^\epsilon\ltimes\mathcal A\).
It is called \emph{admissible} if
the \(\mathcal A\)-action is associative and unital.

For a \(\mathfrak g^\epsilon\)-module \(M\), give
\(\mathfrak g^\epsilon\otimes M\) the diagonal
\(\mathfrak g^\epsilon\)-action and the \(\mathcal A\)-action on
the first factor.  Then
\begin{equation*}
 K_\epsilon(M)=
 \left\{
  \sum_jx_j\otimes v_j\in\mathfrak g^\epsilon\otimes M:
  \sum_j(ax_j)v_j=0\text{ for every }a\in\mathcal A
 \right\}
\end{equation*}
is an \(\mathcal A\mathfrak g^\epsilon\)-submodule, and the
\(A\)-cover
\begin{equation*}
 \widehat M_\epsilon
   :=(\mathfrak g^\epsilon\otimes M)/K_\epsilon(M)
\end{equation*}
is admissible.  We write \(x\boxtimes v\) for the class of
\(x\otimes v\).
There is a natural \(\mathfrak g^\epsilon\)-homomorphism
\begin{equation*}
 \pi_M:\widehat M_\epsilon\longrightarrow M,
 \qquad x\boxtimes v\longmapsto xv,
\end{equation*}
whose image is \(\mathfrak g^\epsilon M\).  In particular, if \(M\) is
nontrivial and simple, then \(\mathfrak g^\epsilon M=M\) and \(\pi_M\)
is surjective.

\begin{theorem}
\label{thm:uniform-cover-finite}
Let \(M\) be a finite Cartan-free
\(\mathfrak g^\epsilon\)-module such that
\(\mathfrak g^\epsilon M=M\).  Then
\(\widehat M_\epsilon\) is finitely generated over
\(\mathbb C[L_0]\).
\end{theorem}

\begin{proof}
Choose \(N\geq2\) as in \Cref{lem:uniform-differentiator-lift}, and
fix \(X,Y\in\{L,G\}\), indices \(k,p\), and \(v\in M\).
The tensor
\[
 w=\sum_{i=0}^{N}(-1)^i\binom Ni
       X_{k-i}\otimes Y_{p+i}v
\]
belongs to \(K_\epsilon(M)\).  Indeed, by
\eqref{eq:uniform-A-on-g}, for every \(s\in\mathbb Z\),
\begin{align*}
 \sum_{i=0}^{N}(-1)^i\binom Ni
 (t^sX_{k-i})Y_{p+i}v
 &=\Omega^{(N)}_{X,Y}(k+s,p)v=0,\\
 \sum_{i=0}^{N}(-1)^i\binom Ni
 (t^s\xi X_{k-i})Y_{p+i}v
 &=\begin{cases}
  \frac12\Omega^{(N)}_{G,Y}(k+s+\epsilon,p)v=0,&X=L,\\
  0,&X=G.
 \end{cases}
\end{align*}
Hence
\begin{equation}
 \sum_{i=0}^{N}(-1)^i\binom Ni
 X_{k-i}\boxtimes Y_{p+i}v=0.
 \label{eq:uniform-cover-difference}
\end{equation}

Since \(M=\mathfrak g^\epsilon M\), the cover is spanned by
\(X_k\boxtimes Y_pv\).  For \(X=L\),
\eqref{eq:uniform-cover-difference} gives
\[
 L_k\boxtimes Y_pv=
 \begin{cases}
  \displaystyle\sum_{i=1}^{N}(-1)^{i+1}\binom Ni
  L_{k-i}\boxtimes Y_{p+i}v,&k\geq N,\\[6pt]
  \displaystyle\sum_{i=1}^{N}(-1)^{i+1}\binom Ni
  L_{k+i}\boxtimes Y_{p-i}v,&k<0.
 \end{cases}
\]
Hence, by induction,
\[
 L_k\boxtimes Y_pv
 \in\sum_{q=0}^{N-1}L_q\boxtimes M,
 \qquad k\in\mathbb Z.
\]
The same argument for \(X=G\) leaves only
\(\epsilon,\ldots,\epsilon+N-1\).  Hence
\begin{equation}
 \widehat M_\epsilon
   =\sum_{q=0}^{N-1}
       \bigl(L_q\boxtimes M+G_{q+\epsilon}\boxtimes M\bigr).
 \label{eq:uniform-cover-bounded-index}
\end{equation}
Choose a homogeneous \(\mathbb C[L_0]\)-basis
\(e_1,\ldots,e_d\) of \(M\).  The diagonal action gives
\begin{equation*}
 X_k\boxtimes f(L_0)v
  =f(L_0-k)(X_k\boxtimes v),
 \qquad f\in\mathbb C[L_0].
\end{equation*}
\begin{samepage}%
Together with \eqref{eq:uniform-cover-bounded-index}, this shows that
\[
 \left\{
  L_q\boxtimes e_j,\ G_{q+\epsilon}\boxtimes e_j:
  0\leq q<N,\ 1\leq j\leq d
 \right\}
\]
generates \(\widehat M_\epsilon\) over \(\mathbb C[L_0]\).
\end{samepage}
\end{proof}

Every admissible \(\mathcal A\mathfrak g^\epsilon\)-module carries
a \(\D\)-action through \(t^{\pm1}\) and \(D\mapsto L_0\).
By \Cref{lem:weyl-finite-length}, if it is finitely generated over
\(\mathbb C[L_0]\), it is free of finite rank and has finite length
as an \(\mathcal A\mathfrak g^\epsilon\)-module.
We can therefore pass from the cover to a simple admissible subquotient.

\begin{proposition}
\label{prop:uniform-minimal-cover}
Let \(M\) be a nontrivial simple finite Cartan-free
\(\mathfrak g^\epsilon\)-module.  There exist a simple admissible
\(\mathcal A\mathfrak g^\epsilon\)-module \(E\) free of finite rank
over \(\mathbb C[L_0]\) and a surjective
\(\mathfrak g^\epsilon\)-homomorphism \(E\twoheadrightarrow M\).
\end{proposition}

\begin{proof}
By \Cref{thm:uniform-cover-finite,lem:weyl-finite-length}, the
cover \(\widehat M_\epsilon\) has a finite composition series of
admissible modules
\[
 0=E_0\subset E_1\subset\cdots\subset E_\ell=\widehat M_\epsilon.
\]
Let \(i\) be the least index such that \(\pi_M(E_i)\neq0\).
Then \(E=E_i/E_{i-1}\) is simple and admissible, and \(\pi_M\) induces
the required epimorphism.  The module \(E\) is finite free over
\(\mathbb C[L_0]\) by \Cref{lem:weyl-finite-length}.
\end{proof}

\section{Proof of the main theorem}
\label{sec:sector-classification}

Let \(\K\) be the Laurent Weyl superalgebra generated by
\(t^{\pm1},\xi,\partial_t,\partial_\xi\), acting naturally on \(\A\).
Identifying \(D\) with \(t\partial_t\), we have
\[
 \K\simeq\D\otimes\operatorname{End}_{\C}(\Lambda(\xi)).
\]
For a \(\D\)-module \(P\), put \(S(P)=P\otimes\Lambda(\xi)\),
with \(P\) placed in even degree and the natural \(\K\)-action.
By graded Morita equivalence, every simple graded
\(\K\)-module is isomorphic to \(S(P)\) or \(\Pi S(P)\) for some
simple \(\D\)-module \(P\).  These modules are strictly simple,
that is, simple after forgetting the grading.
In this notation, the actions defining \(H_{\epsilon,b}(P)\) in
\cref{def:unified-H} become
\begin{align}
 L_m
  &=t^m\left(D+(m-2\epsilon)b
       +\frac{m+2\epsilon}{2}\xi\partial_\xi\right),
       &&m\in\Z,
       \label{eq:unified-H-operator-L}\\
 G_p
  &=t^{p-\epsilon}\left(
       \xi D+2(p-\epsilon)b\xi
       -t^{2\epsilon}\partial_\xi\right),
       &&p\in\Z+\epsilon.
       \label{eq:unified-H-operator-G}
\end{align}

\begin{lemma}
\label{lem:sector-finite-tensor-factorization}
Let \(\mathcal R\) be a unital associative superalgebra, and \(E\) a simple
graded \(\K\otimes\mathcal R\)-module finitely generated over
\(\C[t\partial_t]\).
Then
\[
 E\simeq\Pi^\delta S(P)\otimes V,\qquad\delta\in\Ztwo,
\]
where \(P\) is a simple \(\D\)-module finite free
over \(\C[t\partial_t]\) and \(V\) is a finite-dimensional simple graded
\(\mathcal R\)-module.
\end{lemma}

\begin{proof}
By \Cref{lem:weyl-finite-length}, \(E\) is free over
\(\C[t\partial_t]\) and has finite length as a
\(\D\)-module, hence also as a graded \(\K\)-module.
Choose a simple graded \(\K\)-submodule \(S\subseteq E\).
Then \(S\simeq\Pi^\delta S(P)\)
for some simple \(\D\)-module \(P\) and \(\delta\in\Ztwo\).
Since \(P\) is a direct summand of \(S\) over \(\C[t\partial_t]\),
it is free of finite rank.  Moreover, \(S\) is strictly simple and
\(\K\) has a countable basis, so \(E\simeq S\otimes V\)
for some simple graded \(\mathcal R\)-module \(V\) by
\cite[Lemma~2.2(4)]{XueLu2021}.
Since \(E\) is finitely generated over
\(\C[t\partial_t]\), \(V\) is finite dimensional.
\end{proof}

\subsection{Simple admissible modules}
\label{subsec:simple-admissible-modules}

\begin{proposition}
\label{prop:Ramond-simple-admissible}
Let \(E\) be a simple admissible \(\A\mathfrak g^0\)-module finitely
generated over \(\C[L_0]\).  Then
\[
 E\simeq\Pi^\delta H_{0,b}(P),\qquad
 b\in\C,\quad\delta\in\Ztwo,
\]
where \(P\) is a simple \(\D\)-module finite free over \(\C[D]\).
\end{proposition}

\begin{proof}
By \cite[Propositions~3.1--3.2]{CaiLiuLu2021}, \(E\) is a simple
\(\K\otimes U((t-1)\mathfrak g^0)\)-module with \(D=L_0\).
Thus \Cref{lem:sector-finite-tensor-factorization} gives
\[
 E\simeq\Pi^\delta S(P)\otimes V
\]
with \(P\) as stated and \(V\) a finite-dimensional simple
\((t-1)\mathfrak g^0\)-module.
By \cite[Lemma~3.4(2)]{CaiLiuLu2021}, \(V\) is one dimensional
and annihilated by \((t-1)^2\mathfrak g^0\).
Let \(b\in\C\) be the scalar by which \(L_1-L_0\) acts on \(V\).
Absorbing the parity of \(V\) into \(\delta\) and using the
expressions for \(L_m\) and \(G_m\) in the proof of
\cite[Proposition~3.2]{CaiLiuLu2021}, we obtain
\eqref{eq:unified-H-operator-L}--\eqref{eq:unified-H-operator-G}
with \(\epsilon=0\).
\end{proof}

\begin{samepage}
\begin{proposition}
\label{prop:NS-simple-admissible}
Let \(E\) be a simple admissible \(\A\mathfrak g^{1/2}\)-module
finitely generated over \(\C[L_0]\).  Then
\[
 E\simeq\Pi^\delta H_{1/2,b}(P),\qquad
 b\in\C,\quad\delta\in\Ztwo,
\]
where \(P\) is a simple \(\D\)-module finite free over \(\C[D]\).
\end{proposition}
\end{samepage}

\begin{proof}
By \Cref{lem:weyl-finite-length}, \(E\) is finite free over
\(\C[L_0]\).  Set \(\mathscr C=\operatorname{End}_{\D}(E)\),
where \(D\) acts as \(L_0\).  If \(T\in\mathscr C\) satisfies
\(TE\subseteq L_0E\), then \(Tt=tT\) gives the inclusions
\[
 TE=T(t^nE)=t^nTE\subseteq t^nL_0E=(L_0-n)E,
 \qquad n\in\Z,
\]
which imply \(T=0\).
Thus the action of \(\mathscr C\) on \(E/L_0E\) is faithful,
and \(\mathscr C\) is finite dimensional.

The operators \(t^{-n}L_n-L_0\) belong to \(\mathscr C\), and
\[
 [t^{-m}L_m,t^{-n}L_n]
 =(n-m)t^{-m-n}L_{m+n}+mt^{-m}L_m-nt^{-n}L_n.
\]
The same relation holds after replacing each \(t^{-n}L_n\) by
\(t^{-n}L_n-L_0\).  Applying
\cite[Theorem~15]{BilligFutornyIoharaKashuba2020}, in the purely even
one-variable case, to their induced action on \(E/L_0E\), and
using faithfulness, we obtain polynomiality of \(t^{-n}L_n\) in \(n\).
By interpolation, each \(t^{-n}L_n\), \(n\in\Z\), is a linear
combination of operators \(t^{-j}L_j\) with \(j\geq0\).
Thus every submodule for \(\A\) and the
positive contact subalgebra, generated by
\(L_m\), \(m\geq-1\), and \(G_p\), \(p\geq-1/2\), is
invariant under all \(L_n\).  Since
\[
 [L_m,G_{1/2}]=\frac{1-m}{2}G_{m+1/2},\qquad m\leq-2,
\]
it is also invariant under all \(G_p\).
Hence \(E\) remains simple upon restriction.
By \cite[Lemmas~3.3--3.4]{CaiLu2022}, this restriction is a simple
\(\K\otimes U(\mathcal T)\)-module, where \(\mathcal T\) is
isomorphic to the nonnegative part of the Neveu--Schwarz algebra.
Under the isomorphism of \cite[Lemma~3.3]{CaiLu2022},
\[
 \partial_t=L_{-1},\qquad
 \partial_\xi=\xi L_{-1}-G_{-1/2},\qquad
 L_0=t\partial_t+\frac12\xi\partial_\xi+L'_0.
\]
The operator \(L'_0\) belongs to \(\mathscr C\) and is algebraic.
Since \(t\partial_t\), \((1/2)\xi\partial_\xi\), and \(L'_0\)
commute and the last two are algebraic, \(E\) is
finitely generated over \(\C[t\partial_t]\).

\Cref{lem:sector-finite-tensor-factorization} now gives
\[
 E\simeq\Pi^\delta S(P')\otimes V,
\]
where \(P'\) is a simple \(\D\)-module finite free over
\(\C[t\partial_t]\), and \(V\) is a finite-dimensional simple
graded \(\mathcal T\)-module.  Since \(V\) is generated by a
homogeneous \(L'_0\)-eigenvector and \(\operatorname{ad}L'_0\)
is diagonalizable on \(\mathcal T\), \(V\) is a weight module.  Thus
\cite[Lemma~3.7]{CaiLu2022} implies that \(V\) is one dimensional
and that all positive-degree elements of \(\mathcal T\) act
trivially.  Let \(b\) be the scalar by which \(L'_0\) acts, and
absorb the parity of \(V\) into \(\delta\).
Retaining the \(t\)-action on \(P'\) and setting
\(D=t\partial_t+2b\) defines a simple \(\D\)-module \(P\), still
finite free over \(\C[D]\).  The expressions for \(L_m\) and
\(G_p\) in \cite[(3.1)--(3.2)]{CaiLu2022} yield
\eqref{eq:unified-H-operator-L}--\eqref{eq:unified-H-operator-G}
with \(\epsilon=1/2\), initially for \(m\geq-1\) and
\(p\geq-1/2\).  Polynomiality of \(t^{-m}L_m\) extends
\eqref{eq:unified-H-operator-L} to all \(m\in\Z\).
The commutator with \(G_{1/2}\) then extends
\eqref{eq:unified-H-operator-G} to all \(p\in\Z+1/2\),
proving the assertion.
\end{proof}

\subsection{The classification}
\label{subsec:unified-classification-proof}

\begin{proof}[Proof of \Cref{thm:unified-main}]
Let \(M\) be a nontrivial simple finite Cartan-free module over
\(\mathfrak g^\epsilon\) or \(\widehat{\mathfrak g}^{\,\epsilon}\).
In the latter case, \Cref{prop:uniform-central-charge} gives
\(C=0\), so \(M\) is a \(\mathfrak g^\epsilon\)-module.
By \Cref{prop:uniform-minimal-cover}, there is a simple admissible
module \(E\) finite free over \(\C[L_0]\) with an epimorphism
\(E\twoheadrightarrow M\).  Applying
\Cref{prop:Ramond-simple-admissible,prop:NS-simple-admissible} in the
respective cases gives
\begin{equation}
 \Pi^\delta H_{\epsilon,b}(P)\twoheadrightarrow M,
 \label{eq:H-onto-M-unified}
\end{equation}
where \(P\) is simple and finite free over \(\C[D]\).
By \Cref{prop:unified-H-simple}\textup{(ii)},
\(H_{\epsilon,1/2}(P)\) has no nontrivial simple quotient, so
\(b\ne1/2\).  \Cref{prop:unified-H-simple}\textup{(i)} now
makes the source of
\eqref{eq:H-onto-M-unified} simple, so the epimorphism is an
isomorphism.  Conversely, every listed module is finite Cartan-free
by construction and simple by \Cref{prop:unified-H-simple}\textup{(i)}.
Parameter uniqueness follows from \Cref{cor:unified-H-uniqueness}.
\end{proof}

By \eqref{eq:exceptional-parity-zero}, the simple submodule of
\(H_{\epsilon,1/2}(P)\) is isomorphic to \(\Pi H_{\epsilon,0}(P)\)
and is therefore already included in the classification.

\begin{proof}[Proof of \Cref{cor:finite-generated-simple}]
Let \(T\) be the \(\C[L_0]\)-torsion of \(M\).  For
\(X\in\{L,G\}\) and \(s\) in the corresponding index set, the
relation \([L_0,X_s]=sX_s\) gives
\[
 f(L_0-s)X_sv=X_sf(L_0)v.
\]
Thus \(T\) is a graded submodule, and simplicity gives \(T=0\)
or \(T=M\).  If \(T=0\), the structure theorem over
\(\C[L_0]\) makes both parity components finite free, and
\Cref{thm:unified-main} applies.  If \(T=M\), finite generation
makes \(M\) finite dimensional.
Its generalized \(L_0\)-spectrum is finite, so \(L_{\pm m}=0\)
for all sufficiently large positive integers \(m\).  Taking two
distinct such \(m\) in
\[
 0=[L_m,L_{-m}]=-2mL_0+\frac{m^3-m}{12}C
\]
gives \(L_0=C=0\).  Thus
\[
 X_s=\frac{1}{s}[L_0,X_s]=0,\qquad s\ne0.
\]
In the Ramond case,
\[
 0=[L_1,G_{-1}]=-\frac{3}{2}G_0,
\]
so \(G_0=0\).  Hence the whole algebra acts trivially, and
simplicity leaves the one-dimensional trivial module or its
parity change.
\end{proof}

\subsection{Finite length and extensions}
\label{subsec:finite-length-extensions}

\begin{proposition}
\label{prop:finite-generated-finite-length}
Every \(\mathfrak g^\epsilon\)- or
\(\widehat{\mathfrak g}^{\,\epsilon}\)-module finitely generated
over \(\C[L_0]\) has finite length, with composition factors
among those listed in \Cref{cor:finite-generated-simple}.
\end{proposition}

\begin{proof}
It suffices to consider \(\widehat{\mathfrak g}^{\,\epsilon}\).
By the proof of \Cref{cor:finite-generated-simple}, the
\(\C[L_0]\)-torsion of \(M\) is a finite-dimensional submodule
and every finite-dimensional module is trivial.
Induct on the \(\C[L_0]\)-rank of \(M\), the rank-zero case
being immediate.  After factoring out the torsion, choose a
nonzero graded submodule \(N\) of least rank.  For every
nonzero graded submodule \(W\subseteq N\), the quotient
\(N/W\) is finite dimensional and hence trivial, so
\(\widehat{\mathfrak g}^{\,\epsilon}N\subseteq W\).
Since \(L_0N\ne0\), \(\widehat{\mathfrak g}^{\,\epsilon}N\)
is simple.  The quotient
\(M/\widehat{\mathfrak g}^{\,\epsilon}N\) has smaller rank,
so induction and \Cref{cor:finite-generated-simple} complete
the proof.
\end{proof}

The functors constructed in
\cite[Theorem~3.6]{ChenDaiLiuPei2024} also retain all even
homomorphisms.

\begin{proposition}
\label{prop:H-fully-faithful}
For fixed \(\epsilon\in\{0,1/2\}\) and \(b\in\C\), the
functor \(H_{\epsilon,b}\) is exact and fully faithful.
More precisely, for any \(\D\)-modules \(P,Q\),
\[
 \operatorname{Hom}_{\D}(P,Q)
 \simeq\operatorname{Hom}_{\mathfrak g^\epsilon}
 \bigl(H_{\epsilon,b}(P),H_{\epsilon,b}(Q)\bigr),
\]
via \(f\mapsto f\oplus\overline f\).
A short exact sequence of \(\D\)-modules splits if and only if
its image under \(H_{\epsilon,b}\) splits.
\end{proposition}

\begin{proof}
Write an even homomorphism as \(f\oplus\overline g\).
Compatibility with \(G_{-\epsilon}\) on the odd component
gives \(f=g\), and compatibility with \(G_{1-\epsilon}\)
gives \(ft=tf\).  Compatibility with \(L_0\) on the even
component then gives \(fD=Df\).  Conversely, every
\(\D\)-homomorphism induces an even homomorphism by
\eqref{eq:H-even-components}--\eqref{eq:H-odd-components}.
Exactness follows from \(H_{\epsilon,b}(P)=P\oplus\overline P\),
and the assertion about splitting follows from full faithfulness.
\end{proof}

Varying the parameter \(b\) gives explicit self-extensions.

\begin{proposition}
\label{prop:H-self-extension}
Let \(P\) be a simple \(\D\)-module free of finite rank \(r\)
over \(\C[D]\), and put \(M=H_{\epsilon,b}(P)\), where
\(b\ne1/2\).  Let \(\eta\) be even with \(\eta^2=0\).
Replacing \(b\) by \(b+\eta\) in
\eqref{eq:H-even-components}--\eqref{eq:H-odd-components}
defines an action on
\(E=(P\oplus\overline P)\otimes_{\C}\C[\eta]/(\eta^2)\)
and a nonsplit exact sequence
\begin{equation}
 0\longrightarrow M\longrightarrow E\longrightarrow M
 \longrightarrow0.
 \label{eq:parameter-self-extension}
\end{equation}
The module \(E\) is indecomposable of length two, and each
parity component is free of rank \(2r\) over \(\C[L_0]\).
\end{proposition}

\begin{proof}
The defining relations are polynomial identities in \(b\),
so they remain valid after this substitution.  Multiplication
by \(\eta\) and reduction modulo \(\eta\) give
\eqref{eq:parameter-self-extension}.
Suppose that the projection has an even section, written
\[
 v\longmapsto v+\eta a(v),\qquad
 \overline v\longmapsto\overline v+\eta\overline{c(v)},
 \qquad a,c\in\operatorname{End}_{\C}(P).
\]
Compatibility with \(G_p\) on the odd component gives
\(at^{p+\epsilon}=t^{p+\epsilon}c\).
Taking \(p=-\epsilon\) and \(p=1-\epsilon\) yields
\(a=c\) and \(at=ta\).  Compatibility with \(L_m\) on the
even component therefore gives
\[
 [a,D]=(m-2\epsilon)\operatorname{Id}_P,\qquad m\in\Z,
\]
which is impossible for \(m=0,1\).  Thus the sequence does
not split.  By \Cref{prop:unified-H-simple}\textup{(i)},
its two factors are simple, so \(E\) is indecomposable of
length two.  The sequence splits over \(\C[L_0]\) in each
parity because its quotient is free, and
\Cref{rem:H-Cartan-rank} gives the asserted ranks.
\end{proof}

Together with parity change, this gives a nonsplit finite
Cartan-free self-extension of every module in
\Cref{thm:unified-main}.

\subsection{Ranks and examples}
\label{subsec:rank-dictionary-examples}

A nontrivial simple \(\mathfrak g^0\)-module \(M\) is finite
Cartan-free if and only if it has a finite homogeneous
\(U(\mathfrak h_R)\)-basis.  The forward implication follows from
\Cref{thm:unified-main,rem:H-Cartan-rank}, and the converse from
\(U(\mathfrak h_R)=\C[L_0]\oplus G_0\C[L_0]\).
In this case,
\[
 \operatorname{rank}_{\C[L_0]}M_{\bar0}
 =\operatorname{rank}_{\C[L_0]}M_{\bar1}
 =\operatorname{rank}_{U(\mathfrak h_R)}M.
\]
For \(P\) free of rank \(r\) over \(\C[D]\), both parity
components of \(H_{\epsilon,b}(P)\) have \(\C[L_0]\)-rank
\(r\).  Thus the Ramond rank-one result over
\(U(\mathfrak h_R)\) and the Neveu--Schwarz rank-two result over
\(\C[L_0]\) in \cite{YangYaoXia2020} both correspond to \(r=1\).

For \(r=1\), \eqref{eq:coefficient-matrix-action} gives
\(P_\lambda=\C[D]\) with \(tf(D)=\lambda f(D-1)\),
\(\lambda\in\C^\times\).  By \cref{thm:unified-main}, the
simple modules free of rank two over \(\C[L_0]\) are therefore,
up to parity, \(H_{\epsilon,b}(P_\lambda)\), with uniquely
determined parameters \(\lambda\in\C^\times\) and
\(b\ne1/2\).  For each \(r\geq2\), let
\(g_r(t)=t^{-1}+t^{r-1}\) and equip
\(P_{g_r}=\C[t^{\pm1}]\) with the \(\D\)-action
\[
 t\cdot q(t)=tq(t),\qquad
 D\cdot q(t)=tq'(t)+g_r(t)q(t).
\]
This module is simple and free of rank \(r\) over \(\C[D]\)
by \cite[Example~6.1]{GuoLiuLuZhao2020}, so
\(H_{\epsilon,b}(P_{g_r})\), \(b\ne1/2\), is simple
with each parity component free of rank \(r\) over \(\C[L_0]\).
For \(r=2\), the basis \((1,t)\) and the identity
\(t^2=(D-1)t-1\) give the matrix in
\eqref{eq:coefficient-matrix-action} as
\[
 A(D)=\begin{pmatrix}0&-1\\1&D-1\end{pmatrix},\qquad \det A(D)=1.
\]

\end{document}